%% file: ploop_041113.tex
\documentclass{amsart}
\usepackage{amsthm}
\usepackage{amsfonts}
\usepackage{amsmath,amscd}
\usepackage{graphicx}
\usepackage{color}
\usepackage{amsfonts,amssymb}
\usepackage{mymacros}
\usepackage{mathrsfs}
\usepackage{geometry}

\usepackage{a4wide}

\makeatletter
\renewcommand{\theenumi}{\Roman{enumi}}

\renewcommand{\p@enumii}{\theenumi.}
\makeatother

\newtheorem{neu}{}[section]
\newtheorem{Cor}[neu]{Corollary}
\newtheorem*{Cor*}{Corollary}
\newtheorem{Thm}[neu]{Theorem}
\newtheorem*{Thm*}{Theorem}
\newtheorem*{Observation*}{Observation}
\newtheorem{Prop}[neu]{Proposition}
\newtheorem*{Prop*}{Proposition}
\newtheorem{Lemma}[neu]{Lemma}
\newtheorem*{Lemma*}{Lemma}

\theoremstyle{remark}
\newtheorem*{Rmk*}{Remark}
\newtheorem{Rmk}[neu]{Remark}

\newtheorem*{Claim*}{Claim}

\theoremstyle{definition}
\newtheorem{Def}[neu]{Definition}

\newtheorem*{Ex*}{Example}

\newtheorem*{Qu*}{Question}

\begin{document}

\title{Orderable Contact Structures on Liouville-fillable Contact Manifolds}
\author{Peter Weigel}
\date{\today}
\maketitle
\vspace{.5pc}

\begin{abstract} We study the existence of positive loops of contactomorphisms on a Liouville-fillable contact manifold $\cpairex$.  Previous results (see \cite{EKP06}) show that a large class of Liouville-fillable contact manifolds admit contractible positive loops. In contrast, we show that for any
Liouville-fillable $\cpair$ with $\Dim\Si \geq 7$, there exists a Liouville-fillable contact structure $\xi'$ on $\Si$ which admits no positive loop at all.  Further, $\xi'$ can be chosen to agree with $\xi$ on the complement of a Darboux ball.
\end{abstract}

\vspace{.5pc}
\section{Introduction}
In \cite{EP00}, Eliashberg and Polterovich introduced the notion of \emph{orderability} in contact geometry.  Central to their investigation is the study of the (non-)existence of contractible positive loops of contactomorphisms on a contact manifold $\cpairex$, i.e. loops of contactomorphisms whose $t$-derivative is positive with respect to the contact form.  For example, positive loops of contactomorphisms arise in Riemannian geometry: a P-metric on a closed manifold $M$ induces such a loop on the associated unit cotangent bundle $ST^*M$.  It is a classical result that the existence of such a metric imposes strict topological restrictions on $M$, see \cite{Besse}.

There is also a link between topology and orderability, though it is not as well understood.  Using Givental's nonlinear Maslov index \cite{Givental_The_nonlinear_Maslov_index}, Eliashberg and Polterovich showed that $\RP^{2n-1}$ is orderable.  Eliashberg, Kim, and Polterovich showed that all unit contangent bundles are orderable, but that contact boundaries of 2-subcritical Stein domains (for instance, standard spheres) are not \cite{EKP06}.

This suggests that for fillable contact manifolds, the topology of the filling plays an important role.  We ask a related question: given a non-orderable contact structure on
$\Si$, are all other contact
structures non-orderable as well?  Our main result is the following:

\begin{Thm}\label{Thm_Mthm} Let $\cpair$ be a Liouville-fillable contact manifold with $\Dim\Si$ at least $7$.  Then there exists a Liouville-fillable
contact structure $\xi'$ on $\Si$, agreeing with $\xi$ on the complement of a Darboux ball, which admits no positive loop of contactomorphisms.  In particular, it admits no contractible positive loop of contactomorphisms, thus $\xi'$ is orderable.
\end{Thm}

We now outline the technical framework in which we will be working, and sketch the proof of the above.  Throughout, $(\Si,\al=\la\RE\Si)$ will denote the boundary of a Liouville domain $(W,\om=\d\la)$.  Cieliebak and Frauenfelder have defined an
invariant of such a pair $(\Si,W)$ called Rabinowitz Floer homology
\cite{CF09}.  In this paper, we define the \emph{positive growth rate} $\Ga^+(\Si,W)$ associated to
$\RFH(\Si,W)$.  We show that a positive loop of contactomorphisms can be used to compute $\RFH$.

In \SECT{section_Applications}, we show that $\Ga^+$ detects obstructions to the existence of a positive loop of contactomorphisms.  Specifically, a contact manifold whose filtered $\RFH$ has superlinear growth does not admit one.  This implies, in particular, that a Liouville-fillable contact manifold with $\Ga^+>1$ is orderable.

\begin{Prop} 
Suppose that $(W,\om =\d\la)$ is a Liouville domain with boundary $(\Si,\al = \la\RE\Si)$, which admits a
positive loop of contactomorphisms $\vp$. Then $\Ga^+(\Si,W) \leq 1$. 
\end{Prop}

\begin{Cor}
Suppose that $\cpair$ arises as the contact boundary of a Liouville domain $(W,\om = \d \la)$, with $\Ga^+(\Si,W)>1$.  Then $\cpair$ admits no positive loops of contactomorphisms.  In
particular, it is orderable. 
\end{Cor}

In \SECT{section_proofofMthm}, we complete the proof of \THM{Thm_Mthm} by proving that the contact structure can always be modified locally (using handle attachment
surgeries) so that $\Ga^+(\Si,W') \geq 2$.

\begin{Prop}
If $\Dim\Si \geq 7$, there exists a Liouville domain $(W', \Si')$ which is diffeomorphic to $(W,\Si)$ with
$\Ga^+(\Si',W') \geq 2$.  We can arrange that the Liouville structures agree on the complement of a Darboux ball. 
\end{Prop}

\subsection{Acknowledgements} The author expresses his gratitude to Mark McLean and Alexandru Oancea for several helpful discussions, and to Peter Albers for countless helpful discussions and encouragement.

\section{Preliminary Notions}

We begin by defining some relevant terms and fixing terminology.  Throughout, $\cpairex$ is assumed to be a smooth closed contact manifold.  The Reeb vector field of $\al$ will be denoted $R$.

\begin{Def} 
A positive loop of contactomorphisms (or P-loop) is a smooth map 
\bean
\Map{\vp = \{ \vpt \}_{t \in \R/\Z}}{\R/\Z \X \Si}{\Si} 
\eea
satisfying
\bean
&\DF\vpt\xi=\xi,&&\io_{\dotvp_t}\al>0 \quad \FA t \in \R/\Z, &&& \vp^0=\mr{Id},  
\eea
where we employ the notation $\dotvpt:=\dd{t} \vpt.$ We denote the set of all P-loops by 
\beqn\label{eq_PLoopset} \pls\cpair, \eeq
or simply $\pls$ where no confusion is possible.
\end{Def}

\begin{Rmk} 
For any $g \in C^\INF(\Si),$ one readily observes that 
\beqn
\pls\cpair = \pls(\Si,e^g \al),
\eeq
i.e. although it is essential that $\xi$ be coorientable, the notion of positivity is independent of the choice of contact form within a given coorientation class.
\end{Rmk}

\begin{Def}\label{Def_DP}
Suppose $\vp \in C^\INF(\mr{Cont}\cpair)$ is a path of contactomorphisms.  Following \cite{AFNMI} (and hence in turn \cite{Givental_The_nonlinear_Maslov_index}), we call a point $x \in \mr{Fix}(\vp^t)$ a \emph{discriminant point} of $\vpt$ if 
$$(\pb{\vp_t}\al)_x=\al_x;$$
we call $(x,t)$ a \emph{discriminant pair} of $\vp.$
\end{Def}

\begin{Rmk}
One observes that the notion of discriminant point is independent of the contact form defining $\xi.$  We shall see below that this is not an analytical curiosity, but rather a manifestation of the topological fact that $\vp^t$ is Lefschetz degenerate at $x$, see \DEF{Def_Lefschetz} and \LEM{Lem_mandatorydegeneracy}.
\end{Rmk}

\begin{Def} 
We say that a contact manifold is \emph{Liouville-fillable} if it arises as the boundary of a Liouville domain, i.e. a compact symplectic manifold $(W,\om)$ satisfying
\begin{enumerate} 
\item $\om$ is exact, $\om=\d\la$
\item $(W,\om)$ admits a Liouville vector field $X$; that is, a vector field $X$ satisfying $\LD{X}\om=\om$ whose negative flow is complete
\item $X$ is transverse to $\Si=\partial W$.
\end{enumerate}
\end{Def}

Liouville domains behave well under Cartesian product; that is, if $W_1$ and $W_2$ are Liouville domains, then $W_1 \X W_2$ admits a natural Liouville structure after smoothing the boundary, see \cite{Oancea_Kunneth}.

\subsection{Rabinowitz Floer Homology}
Rabinowitz Floer homology was developed in \cite{CF09} as a fixed-energy analogue of symplectic homology.  It is defined (under suitable assumptions) for
exact convex hypersurfaces in non-compact symplectic manifolds.  Let $(M,\om=\d\la)$ be a non-compact exact symplectic manifold, and let
$\Map{H}{M}\R$ be a smooth autonomous Hamiltonian satisfying the following assumptions:
\begin{enumerate} 
\item $0$ is a regular value of $H$
\item $\Si:=H^{-1}(0)$ is connected and separates $M$ into two components, with $H^{-1}(-\INF,0]$ compact
\item $\d H$ is compactly supported.
\end{enumerate}
The \emph{Rabinowitz action functional} is defined as follows: 
\bea
\mc{A}_H &: C^{\INF}(S^1,M) \X \R \to \R, \\ 
\cp &\mapsto  \int_{S^1}\pb{u}\la - \eta \int_0^1 H(u(t)) \d t. 
\eea
One computes that the critical points of $\mc{A}_H$ are pairs $\cp$ satisfying 
\bea
&\dotu = \eta \hvf{H}(u(t)), \\
&H(u(t)) =0. 
\eea

\nl\noindent
The Rabinowitz Floer chain complex $\mr{RFC}_*(\Si,M,H)$ is a free $\Z_2$ module generated by the critical points of the action functional $\autraf$.  Its homology $\RFH_*(\Si,M)$ is defined by counting Floer
trajectories between critical points, where the grading is given by the Conley-Zehnder index.  We refer to
\cite{CF09}, \cite{CFO09} for more details, and to \cite{AFNMI} for the extension to positive periodic
contact Hamiltonians which we will need here.

\begin{Rmk}\label{Rmk_RFH+}
Recall that Rabinowitz Floer homology is equipped with an action filtration: let 
$$\si(\autraf):=\Set{c}\R{\exists \In{\cp}{\crit{\autraf}} \mbox{ with } \autraf\cp=c}$$
denote the action spectrum of $\autraf$. Let $\ep>0$ be such that
$-\ep \notin \si(\autraf)$, and $c_n$ be an increasing sequence of positive numbers, disjoint from the action spectrum, which
diverges to infinity.  Critical points of the action functional with action in $(-\ep, c_n)$ generate a subcomplex of $\mr{RFC}_*(\Si,M,H)$,
whose homology is denoted 
\beqn
\RFH^{(-\ep,c_n)}_*(\Si,M,H) =: V_n.
\eeq
If $m \leq n$, there is a chain homomorphism induced by inclusion of
complexes 
\beqn
\io_{m,n}: V_m \to V_n.
\eeq
Taken together, these form a directed system whose direct limit we denote by 
\beqn
\RFH^+_*(\Si,M,H):=\lim_{\lra} V_n
\eeq
for $\ep$ chosen sufficiently small. 
\end{Rmk}

We now quote two fundamental results which will be very useful. The first relates $\RFH^+$ with the symplectic
homology of $M$, while the second shows that $\RFH$ depends only on the compact region of $M$ which $\Si$ bounds, which we denote by $W$.  This
makes it possible to define $\RFH^+$ for a Liouville domain $W$.

\begin{Prop}[cf. \cite{CFO09}, Prop 1.4]\label{Prop_CFO_1.4}
$\RFH^+_*(\Si,M)$ is related to $\SH_*(M)$ by the following long exact sequence:
\beq
\cdots \to H^{-*+n}(W,\Si) \to \SH_*(M) \to \RFH^+_*(\Si,M) \to \H^{-*+1+n}(W,\Si) \to \cdots.
\eeq 
\end{Prop}

\begin{Rmk}
Although not explicitly stated, \cite{CFO09} shows that the above long exact sequence is compatible with action filtrations.
\end{Rmk}

\begin{Prop}[cf. \cite{CFO09}, Prop 3.1]\label{Prop_CFO_3.1}
$\RFH^+_*(\Si,M)$ depends only on the exact symplectic manifold $W$ and not on the
ambient manifold $M$. 
\end{Prop}

\begin{Rmk}
We can now define $\RFH^+$ for a Liouville domain $(W,\Si)$ by attaching the symplectization of $\Si$ along $\Si = \partial W$.  \PROP{Prop_CFO_3.1} shows that this is well-defined. In the sequel, we will abuse notation and denote the resulting homology simply by
$\RFH^+_*(\Si,W)$.
\end{Rmk}

We now wish to study the growth of $\RFH^{(-\ep,c_n)}_*(\Si,W)$ as $n$ gets large.  This will allow us to extract more
refined quantitative information than simply the total dimension of the homology, which is often infinite.
We adapt a definition given by
M. McLean in the context of symplectic homology \cite{McLean_Computability}.  Form a sequence of integers
$a(n)$ defined to be the rank of the limiting map:
\beqn
a(n):=\Dim{\mr{Im}(\Map{\io_{n,\INF}}{V_n}{\RFH^+_*(\Sigma,W)})}.
\eeq

\begin{Def} 
The positive growth rate of $\RFH(\Si,W)$ is defined to be 
\beqn
\Ga^+(\Si,W) := \LS_{n\to \INF}\frac{\log(a(n))}{\log(n)}. 
\eeq
\end{Def}

\begin{Rmk} 
A priori, it is not clear that the definition above is, as indicated, independent of choices made.  McLean shows something stronger:
namely, that the growth rate of symplectic homology is defined up to Liouville isomorphism.  \PROP{Prop_CFO_1.4} shows the same is true for
$\RFH^+$. 
\end{Rmk}

\begin{Rmk}
$\Ga^+(\Si,W)$ takes values in $\{-\INF \} \cup [0,\INF]$.  A finite growth rate indicates polynomial growth, while an infinite
growth rate implies super-polynomial (e.g. exponential) growth.
\end{Rmk}

\subsection{The Rabinowitz action functional associated to $\vp$}
Now we will adopt the setup from the introduction, namely that $W$ is a Liouville domain with boundary $\cpair$, and that $\cpair$ supports a
positive loop $\vp$.  We recall the contact Hamiltonian and associated action functional associated to $\vp$.  This is first defined
on the symplectization $S\Si$ of $\Si$, and then extended using cutoff functions to all of 
\beqn
\hat{W}:=W \cup_\Si \left( [1,\INF)\X \Si \right) 
\eeq
in such a way that no additional critical points are introduced \cite{AFNMI}.  We stress that although the critical set
does not depend on the choice of filling, it can happen that the resulting \emph{homology} does.

\begin{Def}\label{Def_Mapsetup} 
The contact Hamiltonian associated to $\vp$ is given by
\bean
&\Map{\myCHt}{M \X S^1}\R \\
&\myCHt(\vpt(x)) = \dfm{\al_{\vpt(x)}}{\dd{t} \vpt(x)}.
\eea
The lifted contact Hamiltonian $\Map{\myHt}{S\Si}\R$ is given by $\myHt = r\myCHt$ and its Hamiltonian flow is given by
$$\phit(r,x)=\left( \frac{r}{\rhot(x)} , \vpt(x) \right),$$
where $\rhot:=\io_R \pb{\vp_t}{\al}.$
\end{Def}

\begin{Def} 
The action functional associated to $\vp$ is defined as
\bean
\raf &: C^\INF(S^1,S\Si) \X \R \to \R, \\	
\cp &\mapsto \int_{S^1} \pb{u}\la - \eta \left. \int_0^1 \left[ \myHet(u(t))-1 \right] \d t. \right.
\eea
The critical points of $\raf$ are pairs $(u,\eta)$ satisfying
\bea\label{eq_cpconditions}
&\dotu=\eta\hvf\myHet, \\
&\myHe(u(1)) = 1.
\eea
\end{Def}

\begin{Rmk}\label{Rmk_soh}
We point out that for $\cp \in \mr{Crit}(\raf)$, $\raf\cp=\eta$, i.e. the action of a critical point is equal to the period. 
\end{Rmk}

\section{Transversality and Admissible P-loops}

To define $\RFH$ for the lifted contact Hamiltonian associated to $\vp \in \pls$, the problem arises that although 
$$\mr{Crit}_{k}(\raf):=\Set{\cp}{\crit\raf}{\raf\cp=k}$$ 
may be identified with $\Si$ by projection for each $k \in \Z$, the identification is somewhat noncanonical.  Therefore, we ensure
that for all critical points with integer period, $r(u(0))=r(u(1))=1$, i.e. that the start and end point lies in $\Si.$

\nl \noindent
\underline{Rescaling the contact form:} Fix $\vp \in \pls$ and define a new contact form $\t{\al}$ for $\xi$ on $\Si$ by setting 
\begin{equation}\label{eq_normalizecontactform}
\t{\al}=\frac{\al}{h_0}.
\end{equation}
We observe that we may view $(\Si,\t{\al})$ as Liouville isotopic to $\cpair$ within $\hat{W}$.  Because $\RFH$ is invariant under such isotopies
\footnote{We were unable to locate a direct proof of this in the literature.  An indirect proof combines analogous results in symplectic homology with the long exact sequences in \cite{CFO09}.  An alternative and more parsimonious approach would be to define $\RFH$ using $\t{\al}$, since all of our results require only that $(\Si,\t{\al})$ is Liouville-fillable.}
, all statements about $\RFH$ do not depend on this choice of rescaling.  Since 
\beqn
\t{h}_0=\t{\al}(\dotvp^0)=1,
\eeq
we obtain the desired property about critical points.  To wit, for $k \in \Z$, periodicity of $\vp$ implies that 
\beqn
\myCH{k}(\vp^k(x))=h_0(x), \quad \rho_{k}(x)=1.
\eeq
Therefore, every $x \in \Si$ corresponds to a critical point, and by the second requirement of \eqref{eq_cpconditions}, 
$$r(u_x(1))=\frac{1}{\t{h}_k(\vp^k(x))}=1.$$

\nl\noindent
The goal of the next two lemmas is to sketch how transversality can be achieved for $\raf$, which here means $\raf$ is Morse-Bott.  Full statements and proofs are given later in this section - here we illustrate the conceptual framework.  The proofs given in \SECT{section_Applications} all rely on the fact that given an \emph{admissible} $\vp \in \pls$, i.e. one for which the Floer homology is well-defined, the growth rate of the homology is bounded above by the growth rate of the chain complex, which is linear.  Given any P-loop, the strategy becomes clear: use its existence to prove the existence of an admissible P-loop.  Assuming for the moment that a generic perturbation yields an admissible \emph{path} of contactomorphisms, care must be taken so that the result is still a loop.  Here the structure of the Rabinowitz action functional is of great help.  By \RMK{Rmk_soh}, the critical values of $\raf$ are simply the periods of the periodic orbits, so it is possible to perturb $\vp$ \emph{in time}, rather than by action.  In addition, imposing the loop condition requires that $\raf$ is Morse-Bott for integer critical values.

\begin{Lemma}\label{Lem_lnonint} Let $\vp$ be a positive loop of contactomorphisms.  Then there exist $\tvp \in \pls$ and a constant $\ep>0$ which depends only on $\vp$ such that 
\begin{itemize} 
\item If $\mr{dist}(t,\Z)\leq \ep,$ $\tvpt=\vpt$
\item Any critical point $\cp \in \mr{Crit}(\mc{A}_{\tvp})$ with $\eta \notin \Z$ is Morse.
\end{itemize}
\end{Lemma}

\begin{proof}[(Sketch)] We focus solely on $\eta \in (0,1)$ and extend by periodicity.  We claim there exists $\ep
> 0$ such that 
$$\si(\raf) \cap ((0,\ep)\cup(1-\ep,1))=\varnothing.$$
Indeed, by compactness of $\Si$ and nonvanishing of $\dotvp$, there exists $\ep>0$ such that $\vpe(x) \neq x$ for any $\eta \in (0,\ep) \cup (1-\ep,1).$  Hence any such $\eta$ is vacuously a regular value of $\raf.$  It is then possible to perturb $\vp$ to $\tvp$ rel endpoints, i.e. so that the perturbation is trivial on the set where $\vp$ is known to be regular.  By choosing the perturbation sufficiently small in $C^1$-norm, the result is still a positive loop of contactomorphisms. We refer to \PROP{Prop_MBatZ} and \PROP{Prop_boom} for details.
\end{proof}

\begin{Lemma}\label{Lem_mb} The Rabinowitz action functional $\raf$ is Morse-Bott when restricted to integral critical values, in the sense that the critical set may be
identified with $\Si=\{ r=1 \}$, and the kernel of the Hessian of $\raf$ may be identified with $T\Si$.
\end{Lemma}

\begin{proof}[(Sketch)]
Here the motivating principle is that since $\vp^k$ is the identity for $k \in \Z$, given a critical point $(u,k)$ and a tangent vector $v \in T_{u(0)} \Si$, one can produce a new critical point by setting 
\beqn
\tilde{u}(t):=\phi^{tk}(\exp_{v}(u(0))),
\eeq
which implies that $T\Si$ is contained in the kernel of the Hessian.  Nondegeneracy of $\om$ and the nonvanishing of $\dotvpt$ then implies that this is the only way the Hessian can vanish.  We refer to \PROP{Prop_MBatZ} for the full argument.
\end{proof}

\input{SS_Ploop_Trans}

\section{Applications}\label{section_Applications}
We now explore some consequences of the preceding.  At the heart of these lies the observation that the existence of a positive loop
of contactomorphisms places strong restrictions on the Rabinowitz Floer homology of a Liouville-fillable contact manifold.  Intuitively speaking,
it can't be ``too big.''

\begin{Prop} 
Suppose $\cpairex$ is a Liouville-fillable contact manifold which admits a positive loop of
contactomorphisms $\vp$.  Then for any filling $W$, $\Ga^+(\Si,W)\leq 1$.
\end{Prop}

\begin{proof} 
The \emph{moral} proof is as follows.  The hypothesis that $\cpair$ admits a P-loop implies that the chain complex
underlying $\RFH$ grows linearly with action.  By \PROP{Prop_admissible}, the existence of \emph{any} P-loop implies the existence of an admissible P-loop, i.e. one for which $\raf$ is Morse-Bott, and for which critical points with non-integer period are Morse.   This implies that the set of critical points with $\eta \in (0,1)$ is finite.  Let us denote its cardinality by $\nu$.  By periodicity, the cardinality of the critical set with $\eta \in (j,j+1)$ is also equal to $\nu$, for any $j \in \Z$.  This implies linear growth of the chain complex, if we focus solely on the summands with non-integer $\eta$-values.

\nl\noindent We wish to make the same claim for the entire chain complex.  To do this, we must specify how to pass from Morse-Bott to Morse.  This is most easily accomplished using the technique of \emph{gradient flows with cascades} developed by Frauenfelder \cite{Frauen03}, see also \cite{CF09} for its use in Rabinowitz Floer homology. To define $\RFH$ with Morse-Bott critical manifolds, one chooses an auxiliary Morse function on each connected component with positive dimension.  There exists a tautological isomorphism of ungraded vector spaces:
$$\bigoplus_{* \in \Z}\mr{RFC}_{\mr{MB},*}^{(a,b)} \left( \Si,W,F,\{f_j \} \right) \cong \Z_2^m \oplus \bigoplus_{i,j} \Z_2 \left< x_{i,j} \right>, $$ 
where the notation indicates that $m$ is the number of points in the zero-dimensional component of the critical manifold, and $x_{i,j}$
is a critical point of the Morse function $f_j$ on the $j^{th}$ component with non-zero dimension.  We point out that the action of $x_{i,j}$ is
equal to 
\bean
\raf^{\mr{MB}}(x_{i,j}) &= \raf(x_{i,j})+f_j(x_{i,j}) \\
&=\eta_j+f_j(x_{i,j}),
\eea 
where $\eta_j$ is the common action on the $j^{th}$ component.

\nl \noindent 
In the present setting, this simplifies greatly.  By construction, for each integer $j$, $$\mr{Crit}_{(j-\ep,j+\ep)}(\raf) \cong \Si$$ 
We choose for each $j$ the same Morse function $\Map{f_j \equiv f}\Si\R$ such that $\|f\|_{L^\INF(\Si)}<\frac{\ep}{2}$.  Recall (cf. \RMK{Rmk_RFH+}) that defining $\RFH^+$ involves the choice of an increasing sequence of regular values which diverges to infinity.  To this end, choose $c_n$ satisfying $c_n \in (n-\ep,n-\frac{\ep}{2})$.  This yields the desired result: let $\nu'$ denote the number of critical points of the Morse function $f$.  Then 
\bean
a(n) &= \mr{rank}(\io_{n,\INF}:\RFH_*^{(-\ep,c_n)}(\Si,W,\myHt,\{ f_j \}) \to \RFH_*^+(\Si,W))    \\
&\leq \mr{dim}_{\Z_2} \left( \RFH_*^{(-\ep,c_n)} \left( \Si,W,\myHt,\{ f_j \} \right) \right) \\
&\leq \mr{dim}_{\Z_2} \left( \mr{RFC}_{\mr{MB},*}^{(-\ep,c_n)} \left( \Si,W,\myHt,\{ f_j \} \right) \right) \\
&=n(\nu' + \nu) \\
\RA \Ga^+ &= \LS_{n \to \INF} \frac{\log(a(n))}{\log n} \leq \LS_{n \to \INF} \frac{\log(n(\nu'+\nu))}{\log(n)} = 1,
\eea
 as required.
\end{proof}

\begin{Cor} Suppose that $\cpairex$ is a Liouville-fillable contact manifold, and that for some filling $W$, $\Ga^+(\Si,W)>1$.  Then $\cpairex$ is orderable. \end{Cor}

\begin{proof} We exploit the fact that the chain complex underlying $\RFH$ is independent of the filling.  If $\cpair$ is non-orderable, it admits a positive loop. $\RFH(\Si,W)$ can be computed as above  using the associated lifted contact Hamiltonian.  Hence $\Ga^+(\Si,W) \leq 1$, a contradiction. \end{proof}

\section{Proof of the Main Result}\label{section_proofofMthm}
We now turn to one of the questions raised in the introduction.  If $\cpair$ is non-orderable, is the same true for all co-orientable contact structures on $\Si$?  We answer in the negative.  The following is \THM{Thm_Mthm} from the introduction.

\begin{Thm}
Let $\cpair$ be a Liouville-fillable contact manifold with $\Dim\Si$ at least $7$.  Then there exists a Liouville-fillable
contact structure $\xi'$ on $\Si$, agreeing with $\xi$ on the complement of a Darboux ball, which admits no positive loop of contactomorphisms.  In particular, it admits no contractible positive loop of contactomorphisms, thus $\xi'$ is orderable.
\end{Thm}

\begin{proof}
The main idea is as follows.  We use Weinstein zero-surgery to form the end-connect sum of $W$ with a
nonstandard ball $B$.  According to Cieliebak \cite{Cieliebak02}, symplectic homology does not change under subcritical
surgery, so
\[
\SH_*\left(W \, \#_e \, B \right) \cong \SH_*(W) \oplus \SH_*(B).
\]
The strategy becomes then to construct $B$ such
that $\Ga^+(B) \geq 2$.  We learned of this line of attack from \cite{McLean_Computability}.

\nl\noindent
According to Seidel \cite{Seidel_Biased}, there exists a contractible 4-dimensional Liouville domain $A$ with
nonvanishing symplectic homology.  We take the cartesian product of $A$ with $DT^*S^1$.  The
result (after smoothing), will be denoted $W_1$, and has a natural Liouville structure.  By work of Oancea
\cite{Oancea_Kunneth},
$$\SH_*(W_1) \cong \SH_*(A) \otimes \SH_*(DT^*S^1).$$
Attach a Weinstein two-handle along a trivially framed isotropic circle in the boundary representing a
generator of $\pi_1$ and call the resulting Weinstein domain $W_2$.  By \cite{Cieliebak02}, 
\begin{equation}
\SH_*(W_2) \cong \SH_*(W_1). \quad \footnote{We point out a slight subtlety here.  When computing the symplectic homology of a manifold with non-contractible orbits, the grading depends on a choice of trivialization for each free homotopy class.  After a Weinstein surgery which kills the homotopy class of an orbit, the resulting trivialization must extend over a spanning disk.  Hence, the above isomorphism should be taken to mean that there exists a choice of grading on $\SH_*(W_1)$ which renders the above a true statement.  For our purposes, this distinction is inconsequential.}
\end{equation}
Now repeat the above, taking $W_3 := W_2 \X DT^*S^1$ and performing Weinstein surgery to obtain $W_4$.  By the h-cobordism theorem, $W_4$ is diffeomorphic to a ball, and $\partial W_4$ is diffeomorphic to a sphere.  Further, 
$$\SH_*(W_4) \cong \SH_*(A) \otimes \SH_*(DT^*T^2).$$
Finally, return to the original pair $(\Si,W)$.  Attach a Weinstein 1-handle to $W \sqcup W_4$.  The resulting pair $(\Si',W_5)$ is diffeomorphic to $(\Si,W)$, with a Liouville structure that agrees with the original one on $(\Si,W)$ on the complement of the ball removed.  

\nl\noindent
To conclude the proof, observe that $\Ga^+(\Si',W_5) \geq 2$. Indeed, $\SH_*(DT^*T^2)$ is isomorphic to the homology of the free
loop space of the torus, which can be computed using the geodesic energy functional.  Alternatively, a direct computation goes as follows: since $\pi_1(T^2)$ is Abelian, $\RFH_*(ST^*T^2,DT^*T^2)$ is indexed by elements of $\pi_1(T^2)$.  Choose the quotient metric from $\R^2/(\Z\oplus \Z)$ so that each of the generators $a,b$ of $\pi_1(T^2)$ has a geodesic representative with minimal length one.  By the choice of metric, the critical manifold associated to the homotopy class
$$\left[ a^i \, b^j \right], \quad (i,j) \in \Z\X\Z$$
is a copy of the torus itself, and hence for each $k\in\Z$,
\bean
\Dim{\RFH_*^{(-\ep,k+\ep)}\left( ST^*T^2,DT^*T^2 \right) } &= 4 \left| \Set{i,j}\Z{i^2+j^2\leq k^2} \right| \\
&= 4N(k) \\
&= 4(\pi k^2 +E(k)), 
\eea
and $E(k)$ satisfies 
\begin{equation}\label{eq_Gaussest}
|E(k)|\leq 2\sqrt{2}\pi k .
\end{equation} 
Here $N(k)$ denotes the number of integer lattice points inside a closed disk of radius $k$ in $\R^2$, $E(k)$ is an error term, and the estimate in \eqref{eq_Gaussest} is due to Gauss.  Taking the log-limit shows that 
$$\Gamma^+(ST^*T^2,DT^*T^2) = 2,$$
as expected.

\nl\noindent
If $\Si$ has dimension greater than $7$, we iterate the procedure above as many times as needed, i.e. the operation of taking the Cartesian product with $DT^*S^1$ and using Weinstein surgery to kill the associated nontriviality in $\pi_1$ gives a well defined map of sets 
\begin{align*}
&\left\{ \text{Liouville domains of dimension } 2n+1\right\} &&\longrightarrow &&&\left\{\text{Liouville domains of dimension } 2n+3 \right\}
\end{align*}
which is monotone increasing with respect to $\Ga^+.$  By this we can obtain a Liouville filling for $S^{2n+1}$ with $\Ga^+(S^{2n+1},B)\geq 2$ for any $n \geq 3$ and thus on any Liouville-fillable contact manifold of the same dimension.
\end{proof}

\appendix

\input{TransSymp_041013}

\bibliography{mybib}{}
\bibliographystyle{plain}
\end{document}

%% file: SS_Ploop_Trans.tex
\subsection{Minimal degeneracy of P-loops}\label{section_ssperturbploop}

Achieving transversality for $\raf$ presents certain difficulties not seen in classical Floer theory, which we summarize here.

\begin{itemize} \item $\raf$ can never be Morse. \end{itemize}
This is to be expected, as we are dealing with the zero-energy level set of a free time action functional, i.e. $\mr{Crit}_{\eta=0}$ will always correspond to $\Si=\{\myH{0}=0\}.$  This is typically handled using Morse-Bott techniques \cite{CF09}.  For a P-loop, the additional complication is that $\mr{Crit}_{\eta \in \Z}\cong \Si$, i.e. $\Si$ appears as a critical manifold infinitely many times.
\begin{itemize} \item The Poincar\'e return map. \end{itemize}
A more serious obstacle is the existence of isolated degeneracies.  Recall (or see below) that a critical point of the classical action functional (corresponding to a fixed point of a Hamiltonian diffeomorphism $\psi$) is nondegenerate if and only if $\psi$ is Lefschetz regular at $x$ (see \DEF{Def_Lefschetz}), i.e. if and only if $1$ is not in the spectrum of the Poincar\'e return map $\psi_*$.  This is satisfied for generic Hamiltonians.  However, for $\phit$ this can never be the case - the $r$-invariance guarantees that $\p{r}$ will always be mapped to itself, see \eqref{eq_dphiformula}.  Here the Lagrange multiplier condition is essential; the second equation in \eqref{eq_cpconditions} ensures that $\p{r}$ cannot lie in the kernel of the Hessian.
\begin{itemize} \item Multiplicity of $1$ as a symplectic eigenvalue. \end{itemize}
Recalling that real eigenvalues of symplectic matrices occur in reciprocal pairs, we see that the algebraic multiplicity of the eigenvalue $1$ is at least $2$.  If the \emph{geometric} multiplicity is at least $2$, we shall see below that regularity is impossible.  This does occur, for example when one studies the lift of the Reeb flow - the critical manifold carries a free action by reparametrization.  For a generic contact form, one can achieve that these correspond to Morse-Bott critical circles.  It is certainly unreasonable to ask that this can be achieved for a nonautonomous contact Hamiltonian using a $C^1$-small perturbation.
\begin{itemize} \item Lefschetz degeneracy of discriminant points. \end{itemize}
The final complication is that any contactomorphism is Lefschetz degenerate at a discriminant point, in the sense that a nonzero vector is fixed by the by the Poincar\'e return map, see \LEM{Lem_mandatorydegeneracy}.  Putting the pieces together, $\phit_*$ at the lift of a discriminant point has the schematic form 
\beqn 
\left[ \begin{array}{cc|c}
1 & \circledast &\mbf{\divideontimes} \\ 
0 & 1 &\mbf{0} \\ \hline 
\mbf{0} & \mbf{\maltese} &(\vpt_*)\RE{\xi} \\ \end{array}
\right]
\eeq
relative to the splitting 
$$T_xW = \mr{span}(\p{r}) \oplus \mr{span}(R) \oplus \xi.$$

\nl\noindent
In light of this, restrictions must be placed on $\vp$ to ensure that the Floer homology is well-defined.  The first is to require that the P-loop be as nondegenerate as possible.

\begin{Def}\label{Def_mindeg} Given $\vp \in \pls$, consider the set of discriminant points 
\beqn 
C:=\Set{(\eta,x)}{\R \X \Si}{\vpe(x)=x, \, \rho_\eta(x)=1}.
\eeq
We say that $\vpt$ is \emph{minimally degenerate} if 
\begin{itemize}
\item For all $(\eta,x)\in C$ with $\eta \notin \Z$, the subspace of $T_xW$ fixed by the Poincar\'e return map has dimension $1$ 
\item Any $k \in \Z$ is a Morse-Bott critical value of $\raf.$
\end{itemize}
Denote the set of all minimally degenerate P-loops by $\pls_\md$.
\end{Def}

\begin{Rmk}\label{Rmk_doublemagic}
We first point out that the second condition in \DEF{Def_mindeg} is always satisfied, see \PROP{Prop_MBatZ}.  Given this, we emphasize that the minimal degeneracy of $\vp$ is then equivalent to $\BAR\vp$ being transverse to the diagonal, see \LEM{Lem_mandatorydegeneracy} and \DEF{Def_dualform}.  This is not inconsequential - it shows that minimal degeneracy is an open condition for paths of contactomorphisms.  Combining this with a straightforward modification of the ideas in \APP{section_GenericSymp}, we obtain that it is a generic property, i.e. that minimal degeneracy holds for an open and dense subset of $C^\INF(\R,\mr{Cont}\cpair).$
\end{Rmk}

\begin{Def}
For any $\vp \in \pls,$ minimally degenerate or otherwise, define the subset of critical points 
$\mr{Crit}_\md(\raf)$ to be the set of those critical points $(u,\eta)$ such that $(\eta,u(0))$ satisfies the criteria of \DEF{Def_mindeg}.  The most economical description for our perturbation method will be the following: define the subset of the action spectrum
$$\si_{\md}(\raf)=\Set\eta{\si(\raf)}{\In{\cp}{\mr{Crit}_{\md}} \text{ for all } \In{u}{\mr{Crit}_\eta(\raf)}}.$$
In this translation,
$$\vp \in \pls_\md \quad\LRA\quad \si_{\md}(\raf)=\si(\raf).$$
\end{Def}

\begin{Prop}\label{Prop_MBatZ}
For any $\vp \in \pls$, $\raf$ is Morse-Bott for integral critical values, in the sense that the critical manifold may be identified with $\Si$ and the kernel of the Hessian with $T\Si.$  Further, for any P-loop $\vp$, there exist $\tvp \in \pls_\md$ and $\ep>0$ such that for all $k \in \Z,$
$$|t-k| \leq \ep \quad\RA\quad \tvpt=\vpt.$$
\end{Prop}

\begin{proof}
By \PROP{Prop_kerlrafconditions}, the kernel of the Hessian $\KLraf$ may be identified with
$$\Set{(a \p{r},\ze_0)}{T_{u(0)}W}{\DF\phie{\ze_0}=\ze_0, \, a \, \myCHe(u(1)) + r(1)\dfm{\d\myCHe}{\ze_0}=0}.$$
For integer periods $\In{k}\Z,$ $\vp^k$ is the identity and $\myCH{k} \equiv 1,$ since the contact form is normalized to the contact Hamiltonian, see \eqref{eq_normalizecontactform}.  Hence
$$a \, \myCHe(u(1)) + r(1) \dfm{\d\myCHe}{\ze_0}=0 \quad\RA\quad a=0,$$
and since $\DF{\vp^k}{\ze_0}=\ze_0$ for all $\ze_0 \in T\Si,$ we obtain that $\KLraf$ is precisely the tangent space to the critical manifold $\Si$, hence
$$\Z \in \si_\md(\raf) \quad \FA \vp \in \pls;$$ 
equivalently, $\raf$ is Morse-Bott for integral critical values.  This proves the first statement.

\nl \noindent
Next we show that there exists $\ep>0$ such that for any $\In{k}\Z,$ 
$$0<|\eta-k| \leq \ep \quad\RA\quad \mr{Crit}_\eta(\raf)=\varnothing.$$  
This is not a feature unique to positive contactomorphisms, but holds for any smooth nonvanishing vector field on a compact manifold.  

\begin{Claim*} Let $X$ be a smooth compact manifold, and let $Y_t,$ $t \in [-a,a]$ be a smooth nonvanishing time-dependent vector field, with flow $\psi^t$ and $\psi^0$ the identity map.  Then there exists $\ep>0$ such that $\mr{Fix}(\psi^t)=\varnothing$ for $0<t<\ep.$
\end{Claim*}

\begin{proof}
Fix an auxiliary metric $g$ on $X$, and denote its injectivity radius by $\delta_g$.  Then \begin{align*}
|t|<\ep_1 &:= \frac{\delta_g}{\|Y_t\|_\INF} \RA \psi^t(x) \in B_{\delta_g}(x). \\
\intertext{Assume} 
0<|t|<\ep&:=\min \left(\ep_1,\frac{\min_{X \X [-\ep_1,\ep_1]}(Y_t)}{\| \nab_t Y_t\| _{g,L^\INF}} \right),
\end{align*}
and using Euclidean coordinates for $B_{\delta_g}(x)$, we obtain 
\begin{align} 
\notag 0 &< \left| tY_t \right| - \frac{t^2}{2} \left\| \nab_t Y_t \right\|_{g,L^\INF} \\
\notag &\leq \left| tY_t \right|_g-\int_0^t s \, \left| \nab_s Y_s \right| _g \d s \\
\notag &\leq \left| tY_t \right|_g-\left| \int_0^t s\nab_s Y_s \d s \right|_g \\
\notag &=\left| \left| tY_t \right|_g-\left| \int_0^t s\nab_s Y_s \d s \right|_g  \right|\\
\notag &\leq \left| tY_t-\int_0^t s\nab_s Y_s \d s \right|_g \\
&= \left| \int_0^t Y_s \d s \right|_g = \mr{dist}_g(x,\psi^t(x)). 
\qedhere 
\end{align}
\end{proof}

Using this, there exists $\ep>0$ such that 
$$
0<\mr{dist}(c,\Z)<\ep \RA c \notin \si(\raf).
$$
For any $\tilde{\vp}$ such that for $t \in [0,1]$ there holds: 
\begin{itemize} 
\item $\tvpt=\vpt$ if $t \in [0,\ep] \cup [1-\ep,1]$
\item $\BAR{\tvp}\RE{t \in [\ep,1-\ep]}$ is transverse to the diagonal, see \DEF{Def_Lefschetz};
\end{itemize} 
we can extend by periodicity to obtain a periodic loop of contactomorphisms.  The transversality in the second requirement can be achieved rel endpoints and such that $\tvp$ and $\vp$ are arbitrarily close in $C^1$ norm, meaning that $\tvp$ can be taken to be a P-loop.  By \LEM{Lem_mandatorydegeneracy} and \RMK{Rmk_doublemagic}, $c \in \si_\md(\mc{A}_{\tvp})$ for all $c \notin Z$, and hence $\tvp \in \pls_\md.$ 
\end{proof}

In the sequel, we will assume absent proviso that any P-loops mentioned are minimally degenerate.

\subsection{Analysis of the Hessian, part 1: revisiting the classical action functional}

Suppose now that $M$ is symplectic and $\om$ is the symplectic form.  Assume for simplicity that $\om=\d\la$ is exact.  For a Hamiltonian $\Map{\ham{H}{t}}{M}{\R},$ define the classical action functional 
\bea
\caf &\co C^{\infty} \left(S^1,M \right)  \to \R \\
u &\mapsto \int_{S^1} \pb{u}{\la} - \int_0^1 \ham{H}{t}(u(t))  \d t. 
\eea
Critical points of the action functional are closed orbits of the flow $\phit$ of the Hamiltonian vector field.  The differential of the action functional 
\beqn
\Map{\caf(u)}{L^2(\pb{u}{TM})}{\R}
\eeq
vanishes at $u$ if and only if 
$$
\int_0^1 \dfm{\om}{\dot{u}-\hvf{H_t},\mu} \d t=0 \quad \FA \In{\mu}{L^2(\pb{u}{TM})},
$$
i.e. if and only if $u \in \crit\caf.$  For such $u$, we first define
\begin{align*}
&\Map{D_u}{L^2(\pb{u}{TM})^{\otimes 2}}{\R}, \\
&(\ze,\mu)\mapsto \int_0^1 \ze \dfm{\om}{\dot{u}-\hvf{H_t},\mu} \d t,
\end{align*}
which is well-defined since $u$ is a critical point.  We then define 
\beq\label{eq_lcafdef}
\Map{\lcaf}{L^2(\pb{u}{TM})}{\left( L^2(\pb{u}{TM}) \right)^*} \to L^2(\pb{u}{TM}) 
\eeq
where the first map assigns to $\ze$ the linear functional 
\beqn
\lcaf(\ze)\co \mu \mapsto \lcaf(\ze,\mu),
\eeq
and the second is the $\om$-duality isomorphism.

\begin{Rmk}\label{Rmk_pseudohessian}
We call $\lcaf$ a \emph{pseudo-Hessian} because it is anti-symmetric.  We could of course insist on symmetry by defining an $L^2$-metric associated to $\om$, but this requires auxiliary choices (say, of an almost complex structure).  We can still refer to the kernel of \eqref{eq_lcafdef} as the \emph{kernel of the Hessian}, since this is independent of choices.
\end{Rmk}

\begin{Lemma}\label{Lem_magic6} 
If $u \in \crit\caf$ and $\ze \in \Ker{\lcaf}$, then $\ze(0) = \ze(1).$
\end{Lemma}

\begin{proof} 
We assume far too much, but state it this way for reference.  Consider the circle as $S^1 = \R/\mbb{Z}.$ Since $u$ is a map from the circle, any vector field in $\pb{u}{TM}$ is the derivative at $s=0$ of a map 
\beqn
\Map{v(s,t)}{(-\ep,\ep) \X S^1}{M},  \quad v(0,t) = u(t) .
\eeq
Since $v(s,t)$ is a loop for all $s$, 
\begin{align}
\notag v(s,0)&=v(s,1) \quad \FA s \in (-\ep,\ep) \\
&\RA \ze(0) = \p{s} v(0,0)  = \p{s} v(0,1) =\ze(1) . 
\qedhere
\end{align}
\end{proof}

We now state for reference an elementary, but important, Lemma.  The proof is omitted.

\begin{Lemma}\label{Lem_magic4}
Given a symplectic manifold $(M,\om)$ and a symplectic isotopy $\psi^t$, suppose $u:[0,1] \to M$ satisfies $\dotu=\dot{\psi}^t$ and 
$\Map{v}{(-\ep,\ep) \X [0,1]}{M}$ satisfies $v(0,t)=u(t).$  Then 
\begin{align*}
\dd{s}\RE{s=0} &\dfm{\om}{X-\p{t}v, \mu} = 0 \FA \mu \in \pb{u}{TM} \\
&\LRA \p{s}v(0,t) =\DF{\psi^t}{\p{s}v(0,0)}.
\end{align*}
If pointwise vanishing is replaced by 
$$ \dd{s}\RE{s=0} \int_0^1 \dfm{\om}{X-\p{t}v, \mu} \d t = 0 \quad \FA \mu \in L^2(\pb{u}{TM}),$$
the conclusion remains true if the equality is taken in $L^2.$
\end{Lemma}

\begin{Cor}\label{Cor_magic5}
 The kernel of the Hessian of $\caf$ at a critical point $u$ is identified with 
$$\Set{v}{T_{u(0)}M}{\DF{\phi^1}{v}=v}.$$ 
\end{Cor}

\begin{proof}
\LEM{Lem_magic4} plus \LEM{Lem_magic6} show the following:
\begin{align} 
\notag \Ker{\lcaf}  &= \Set{\ze}{L^2(\pb{u}{TM})}{\ze(t)=\DF{\phit}{\ze(0)}} \\
&\cong \Set{v}{T_{u(0)}M}{\DF{\phi^1}{v}=v}
\qedhere
\end{align}
\end{proof}

\subsection{Analysis of the Hessian, part 2: the kernel of $\Lraf$}

We now extend the analysis of the classical action functional to the Rabinowitz action functional.  We assume in this section that the symplectization of $\cpair$ is embedded in a symplectic manifold $W$, and use coordinates $(r,x)$ as before.  However, we do not require here that $\Si$ is Liouville fillable or admits a P-loop.  Assume
\beqn 
\Map\vp{[0,1]}{\mr{Cont}\cpair}
\eeq  
is a smooth positive \emph{path} of contactomorphisms.  We recall for the reader's convenience a brief glossary of the terminology used.

\begin{Def}\label{Def_assobjects}
To any positive path $\vp$, define
\begin{equation}\label{assfunctions}
\begin{aligned}
&\Map{\rho_t}{\Si}{\R},   && \rho_t:=\io_R (\pb{\vp_t}\al)   \\
&\Map{h_t}{\Si}{\R},      && h_t= \io_{\dotvp_t} \al  \\
&\Map{\myHt}{W}{\R},      && \myHt(r,x)=r\, h_t(x)-1  \\
&\Map{\phit}{W}{W},       && (r,x)\mapsto \left(\frac{r}{\rho_t(x)},\vpt(x)\right). 
\end{aligned}
\end{equation}
We point out that $\myHt$ generates $\phit$ as an exact Hamiltonian symplectomorphism of $W.$
\end{Def}
\begin{Def}\label{Def_kraf} 
The Rabinowitz action function functional associated to the data $(\Si,W,\vp)$ is
\begin{equation}\label{eq_krafdef}
\begin{aligned}
\raf &: C^{\infty} \left(S^1,W \right)  \X \R \to \R \\ 
(u,\eta)  &\mapsto \int_{S^1} \pb{u}{\la} - \eta \int_0^1 \myHet (u(t))  \d t. 
\end{aligned}
\end{equation}
The critical set is found to be
\beq\label{eq_rafcritformal}
\crit{\raf} = \Set{(u,\eta)}{C^{\infty}(S^1,W) \X \R}{u(t)=\phi^{\eta t}(u(0)), \, \myHe(u(1))=0},
\eeq
and can be identified with
\beq\label{eq_rafcrituseful} 
\crit{\raf} \cong \Set{(r,x,\eta)}{W \X \R}{\vp^{\eta}(x)=x,\, \rho_{\eta}(x)=1,\, r=\frac{1}{h_{\eta}(x)}}.
\eeq
At a critical point, the differential is defined as follows:
\beqn
\dfm{\d\raf}{\ze,b} =\dd{s}\RE{s=0} \raf (v(s,t) ,\eta(s)),
\eeq 
where $\p{s} v(0,t)  = \ze,$ $v(0,t) =u(t) ,$ $\eta' (0) = b,$ and $\eta (0) =\eta.$  The pseudo-Hessian 
\beqn 
\Map{\Lraf}{L^2(\pb{u}{T\Si}) \X \R}{L^2(\pb{u}{T\Si}) \X \R} 
\eeq 
and the associated $\KLraf$ are defined as in \RMK{Rmk_pseudohessian}.
\end{Def}

\begin{Lemma}\label{Lem_bzero} 
Given $\cp\in\mr{Crit}(\raf)$ and paths 
\beqn
\Map{\eta(s)}{(-\ep,\ep)}\R,\quad \Map{x(s)}{(-\ep,\ep)}{W}
\eeq
satisfying $x(0)=u(0),$ $\eta(0)=\eta$, define
\beqn
v(s,t) := \phiest(x(s)).
\eeq
Set $\ze:=\p{s}v(0,t),$ $b:=\eta'(0).$  Then if $\,\In{(\ze,b)}\KLraf,$ $b=0.$ 
\end{Lemma}

\begin{proof} 
For notational convenience, set $x:=u(0)$ and $\ze_0:=\ze(0)$.  Although $v(s,t)$ is not necessarily a variation through loops, the assumption that $(\ze,b) \in \KLraf$ implies that $\p{s} v(0,0) =\p{s} v(0,1) .$  A short computation shows 
\beqn
\p{s} v = \eta'(s) \, t \dotpest + \DF\phiest{x'(s)} ,
\eeq
and hence the loop condition reads
\begin{align}\label{eq_b=0}
\notag 0 &= \p{s} v(0,1)  - \p{s} v(0,0)  \\
\notag &= \eta'(0)  \dotpe + \DF{\phie}{x'(0)} - 0 - \DF{\phi^0}{x'(s)} \\
\RA \ze_0 &=b \dotpe + \DF{\phie}{\ze_0}. \\
\intertext{This last equation follows from the definitions and the fact that }
\notag \phi^0 = \mr{Id} &\RA \DF{\phi^0}{\ze_0} =\ze_0 . \\
\intertext{By \eqref{eq_rafcrituseful} there holds }
\notag \rho_{\eta} (x) =1 &\RA \pbp{\phi_{\eta}}{\al}(x) =\pbp{\vp_{\eta}}{\al}(x) =\al(x)  \\
\notag &\RA \dfm{\al}{v} =\dfm{\al}{\DF{\phie}{v}}, \quad \FA \In{v}{T_xW}. \\ 
\intertext{Apply $\al$ to \eqref{eq_b=0} to obtain that }
\notag \dfm{\al}{\ze_0} &=\dfm{\al}{b \dotpe} +\dfm{\al}{\DF{\phie}{\ze_0}} \\
\notag &=b \dfm{\al}{\dotvpe} + \dfm{\pbp{\phi_{\eta}}{\al}}{\ze_0} \\
\notag &=b \, h_{\eta}+\dfm{\al}{\ze_0}. \\
\intertext{The assumption that $\vp^t$ is a positive path is equivalent to $h>0$. Hence }
\notag \dfm{\al}{\ze_0} &=  b \, h_{\eta}+\dfm{\al}{\ze_0} \\
\notag &\RA  b \, h_{\eta} =0 \\
&\RA  b=0. 
\qedhere
\end{align}
\end{proof}

\begin{Lemma} 
Suppose $(\ze,b) \in \KLraf$.  Then $b=0.$ 
\end{Lemma}

\begin{proof} 
First assume that $b=0.$  Observe that linearization of the Rabinowitz action functional with both time vectors equal to zero, i.e. 
$$\Lraf(\ze,0)(\hat{\ze},0) ,$$
is simply the linearization of the classical action functional.  Hence \COR{Cor_magic5} applies, and we obtain that 
\beqn
(\ze,0) \in \KLraf \RA \ze(t) =\DF{\phiet}{\ze(0)}.
\eeq

\nl\noindent
Now given $ (\ze,b)  \in \KLraf ,$ set $x := u(0) $ and define 
\beqn
v(s,t) := \phi^{\eta(s) t} (x(s))
\eeq
as in \LEM{Lem_bzero}.  A priori, $\p{s} v(0,0)$ may not equal $\p{s} v(0,1) .$  We claim that 
\begin{equation}\label{eq_calclemma}
\pd{s}\RE{s=0}\dfm{\om}{\p{t} v -\eta(s)\hvf{\myHest},\mu} =0 \quad \FA \In{\mu}{\Ga(\pb{u}{TW})}.
\end{equation}
Differentiating the expression in \eqref{eq_calclemma} is simple - the difference quotient is zero for all $s$!  Denoting $\p{s} v(0,t)$ by $\t{\ze},$ we then have that 
\bea
&\              &\Lraf(\t{\ze},b)   &= \Lraf(\ze,b)=0 \\
&\RA    &\Lraf(0,b)             &=\Lraf(-\t{\ze},0)  \\
&\RA    &\Lraf(\ze,b)          	&=\Lraf(\ze-\t{\ze},0)=0 , 
\eea
which reduces to Step 1.  Since $\ze(0) =\t{\ze}(0)$, Step 1 shows that $$\ze\equiv\t{\ze}.$$  This implies that $ (\t{\ze},b) \in \KLraf $ (i.e. we now know that $\t{\ze}$ is periodic), and hence by \LEM{Lem_bzero}, $b=0.$ 
\end{proof}

The following proposition summarizes our findings thus far.

\begin{Prop}\label{Prop_kerconditions}
If $(\ze,b)  \in \KLraf,$ then 
\begin{enumerate}
\item $b=0$
\item $\ze(t) = \DF{\phiet}{\ze(0)}$
\item $\ze(0) = \ze(1)  = \DF{\phie}{\ze(0)}$ \label{list_kerconditionsthree}
\item If we write write $ \ze(0) =\ze_r \p{r} + \ze_{\Si} $ according to the splitting  \label{list_kerconditionsfour}
\beqn
TW=T \R \oplus T\Si,
\eeq
then \ref{list_kerconditionsthree} implies that 
\beqn 
\dfm{\d \rho_{\eta}}{\ze_{\Si}} =0, \quad \DF{\vpe}{\ze_{\Si}} = \ze_{\Si}. 
\eeq 
\end{enumerate}
\end{Prop}

\begin{proof} We have proven all of the above except \ref{list_kerconditionsfour}.  We use the definition of $\phi$ (see \DEF{Def_Mapsetup})
plus the fact that for a critical point $\cp$ there holds
\beqn
u(t)=\phi^{\eta t}(u(0)), \quad \rho_{\eta}(u(0))=1
\eeq
to compute 
\begin{equation}\label{eq_dphiformula}
\begin{aligned}
\DF{\phie}{\ze(0)} &= \frac{1}{\rho_{\eta}}\,\ze_r \p{r} + \DF{\vpe}{\ze_r \p{r}} - \frac{r \dfm{\d\rho_{\eta}}{\ze_{\Si}}}{\rho_{\eta} ^2 } \p{r} + \DF{\vpe}{\ze_{\Si}} \\
&= \ze_r \p{r} + 0 - r \dfm{\d\rho_{\eta}}{\ze_{\Si}} \p{r} + \DF{\vpe}{\ze_\Si} \\
&= \left( \ze_r - \dfm{\d\rho_{\eta}}{\ze_{\Si}} \right) \p{r} + \DF{\vpe}{\ze_\Si}. 
\end{aligned}
\end{equation}
\ref{list_kerconditionsfour} follows by setting the last expression in \eqref{eq_dphiformula} equal to $\ze_r \p{r} + \ze_{\Si}.$
\end{proof}

\begin{Rmk} 
We stress that \PROP{Prop_kerconditions} gives necessary conditions for $(\ze,b)$ to lie in $\KLraf$.  It is equivalent to saying that $\Lraf(\ze,b)(\mu,0)=0$, i.e.
\beqn
\dd{s}\RE{s=0} \left[\int_0^1 \dfm{\om}{\p{t} v-\eta(s) \hvf{\myHest},\mu}  \d t\right] = 0, \quad \FA \mu \in L^2 (\pb{u}{TW}).
\eeq 
However, the space of vanishing sections may be reduced by requiring that they must also lie in the kernel of the second component.
\end{Rmk}

\begin{Rmk} 
In order to simplify things, we will not compute the full kernel of the second component, i.e. determining all $(\ze,b)$ such that  
\beq\label{eq_comptwo}
\Lraf(\ze,b)(0,\hat{b})=0, \quad \FA \In{\hat{b}}{\R}.
\eeq
Using the relationship 
\beqn 
A \cap B= A \cap \left(A\cap B \right),
\eeq
it suffices to determine those solutions of \eqref{eq_comptwo} which also meet the criteria of \PROP{Prop_kerconditions}. \end{Rmk}

\begin{Prop}
$(\ze,0) \in \KLraf$ if and only if $\ze$ satisfies the conditions of \PROP{Prop_kerconditions} and 
$\dfm{\d\myHe}{\ze(1)} = 0. $
\end{Prop}

\begin{proof} We compute the differential with respect to $\eta$ and find
\bea
\dd{s}\RE{s=0} &\left[\int_{S^1} \pb{u}{\la} - \eta(s) \int_0^1 \myHest (u(t)) \d t  \right] \\
&= -\dd{s}\RE{s=0} \left[  \eta(s) \int_0^1 \myHest (u(t)) \d t \right] \\ 
&= -\dd{s}\RE{s=0} \left[ \int_0^{\eta(s)} \myH\tau \left( u \left( \frac { \tau } { \eta(s) } \right) \right) \d \tau \right] \\
&= -\eta'(s) \left[  \myHe (u(1)) - \int_0^{ \eta(s) } \frac { \tau }{ \eta(s) ^2 } \dfm{\d\myH\tau}{\dot{u} \left( \frac{\tau}{\eta(s)} \right)} \d \tau \right] \\
&= -\eta'(s) \left[ \myHe(u(1)) - \int_0^1 t \dfm{\d\myHest}{\dot{u}(t)} \d t \right]. 
\eea
To complete the proof, it suffices to show that 
$$\int_0^1 \ze t \dfm{\d \myHet}{\dot{u}(t)} \d t$$
vanishes.  Precisely as in \LEM{Lem_bzero}, we choose 
\beqn
v(s,t):=\phi^{\eta(s)t}(x(s)), \quad \mbox{where } \p{s} v(0,t)=\ze,\,\eta'(0)=0
\eeq
and compute
\beq\label{eq_secondcomp}
\begin{aligned}
\int_0^1 t \dfm{\d\myHest}{\p{t} v} \d t &= \int_0^1 t \dfm{\d\myHest}{\eta(s) \hvf{\myHest}} \d t \\
&=\eta(s) \int_0^1 t \dfm{\d\myHest}{\hvf{\myHest}} \d t \\
&=0.
\end{aligned}
\eeq
The last equality in \eqref{eq_secondcomp} uses the fact that for any Hamiltonian $H$,
\beqn
\iota_{\hvf{H}} \d H =-\dfm\om{\hvf{H},\hvf{H}} =0.
\eeq
Thus since the expression in \eqref{eq_secondcomp} is identically zero, its derivative is zero.
\end{proof}

\begin{Prop}[Summary]\label{Prop_kerlrafconditions}
At a critical point $\cp$ of the Rabinowitz action functional, the kernel of the Hessian is identified with
\bean
\KLraf &= \Set{v}{T_{u(0)}W}{\DF{\phie}{v}=v,\, \dfm{\d\myHe}{v}=0} \\
&\cong \Set{a \p{r} + v_\Si}{T_{u(0)}W}{\DF{\vpe}{v_\Si}=v_\Si, \, \dfm{\d \rhoe}{v_\Si}=0}
\eea
\end{Prop}

\begin{proof}
The identification follows from an interpolating step.  Given $a\p{r}+v_\Si,$ we know that the condition of the second line is necessary.  To see that it is sufficient to obtain a bijection of sets, observe that 
$$  \dfm{\d\myHe}{v}= a h_\eta(u(1)) + r(u(1)) \dfm{h_\eta}{v_\Si}.$$
Using the positivity of $h_\eta$, we see that it is always possible to solve for $a$ uniquely.
\end{proof}

\subsection{The set of admissible P-loops}

The final step is to show that for a generic P-loop $\vp,$ the Rabinowitz action functional is Morse for noninteger critical values.  We begin by describing summarizing equivalent conditions for this to be true, shown in the last section.

\begin{Rmk}\label{Rmk_Equivconditions}
The following are equivalent for $c \in \si(\raf)\setminus\Z.$
\begin{enumerate}
\item $c$ is a Morse critical value of $\raf$
\item $\BAR{\phi}$ is transverse to the diagonal at $t=c$
\item For any $u$ with $(u,c)\in \crit\raf$, $E_1(\phie,u(0))$ is spanned by $\p{r}$
\item For any $u$ with $(u,c)\in \crit\raf$, $\BAR{\vp}$ is transverse to the diagonal at $t=c$ and $E_1(\vp,c)$ is transverse to $\Ker{\d \rho_c}$ at $u(0)$.
\end{enumerate}
Here we use the terminology $E_1$ and $\BAR{\vp}$ from \DEF{Def_Lefschetz}.
\end{Rmk}

We now show that this holds for a generic P-loop.

\begin{Prop}\label{Prop_boom}
The conditions in \RMK{Rmk_Equivconditions} are equivalent to the statement that $\BAR{\vp}^c$ is transverse to the diagonal and $\d \rho_c(x) \neq 0$ for any $x \in \Si$ with $\vp^c(x)=x$ and $\rho_c(x)=1.$  This is satisfied for a generic $\vp \in \pls.$
\end{Prop}

\begin{proof} 
We wish to show that for a generic $\vp \in \pls$, $E_1(\phi^c_*)$ is one-dimensional at any critical point.  This is equivalent to
\beq\label{eq_cok}
\Dim{\mr{coker}(\phi^c_*-I)}=1.
\eeq 
Since $\BAR{\vp}^{c}$ is transverse to the diagonal, $E_1(\vp^c_*)$ is one-dimensional, and since $\rho_c=1$ for a critical point, we obtain that 
$$\mr{coker}(\vp^c_*-I)=\mr{span}(R).$$
This implies as well that the contact distribution $\xi$ lies in the image of $\phi^c_*-I.$  Combining this with \eqref{eq_dphiformula}, we see that 
$$R \in \mr{coker}(\phi^c_*-I),$$
and hence the condition that the cokernel is one-dimensional is equivalent to the condition that $\p{r}$ lies in the image of $\phi_*^c-I.$  Using \eqref{eq_dphiformula} again, this is equivalent to the nonvanishing of $\d \rho_c.$  If we split the tangent space
$$T_{u(0)}\hat{W}=\mr{span}(\p{r}) \oplus \mr{span}(R) \oplus \xi,$$
any symplectomorphism whose Jacobian always has the form 

\[
\begin{aligned}
\renewcommand{\arraystretch}{4}
\left[\begin{array}{cc|c}
\left. \renewcommand{\arraystretch}{1} \begin{array}{c} b \\ 0 \\ \end{array} \right.
& \left.\renewcommand{\arraystretch}{1}\begin{array}{c} a \\ \frac{1}{b} \\ \end{array} \right.
& \left.\renewcommand{\arraystretch}{1}\begin{array}{c} \mbf{w} \\ \mbf{0} \\ \end{array}\right. \\ \cline{1-3}
\mbf{0} & \mbf{v} & \mbf{A} \in \sg \\
\end{array}\right]
&\quad\mbox{where } & \left. \begin{array}{c}  b>0, \, a>0, \\ \mbf{v} \in \mr{Mat}_{1,2n}(\R), \\ \mbf{w} \in \mr{Mat}_{2n,1}(\R). \\  \end{array} \right.
\end{aligned}
\]	
is the lift of a contactomorphism.  At a discriminant point, $b=1$ and condition \eqref{eq_cok} will be satisfied if $\dim(E_1(\vpe_*))=1$ and $a$ or $w$ is nonzero.  Using the techniques of \APP{section_GenericSymp}, if we perturb using any map whose Jacobian locally has the form
\[
\exp(JS), \quad S= 
\left[ 
\begin{array}{cc|c}
0 & 0 & \mbf{0} \\ 
0 & 0 & \circledast \\ \hline 
\circledast & \mbf{0} & \mbf{S'} \\ \end{array} 
\right],
\] 
where $J$ interchanges $\p{r}$ and $R$, the result generically has $\ker(\tp_*-I)=\mr{span}(\p{r})$ at any lifted discriminant point.
\end{proof}

\begin{Def}\label{Def_admissible}
We say $\vp$ is \emph{$\RFH$-admissible} if it is minimally degenerate and if for any discriminant pair $(x,\eta)$ with $\eta \notin \Z$, one (and hence all) of the following equivalent conditions holds:
\begin{enumerate}
\item $\d\rhoe(x)\neq0$
\item $\DF{\vpe}{R} \neq R$
\item $\pb{\vp_\eta}{\d \al} \neq \d \al.$
\end{enumerate}
We denote the set of admissible P-loops by $\pls_a.$
\end{Def}

\begin{Prop}\label{Prop_admissible}
Given any admissible P-loop $\vp \in \pls_a,$ $\raf$ is Morse Bott, hence the Rabinowitz Floer homology is well-defined.  More precisely, $\raf$ is Morse-Bott with critical manifold $\Si$ for integer critical values, and Morse for noninteger critical values.  Further, $\pls_a$ is open and dense in $\pls$, and given any P-loop $\vp$, there exists $\tvp \in \pls_a$ which agrees with $\vp$ in a uniform neighborhood of the integers and can be chosen arbitrarily close to $\vp$ in the $C^\INF$ norm.
\end{Prop}

\begin{proof}
Combining \PROP{Prop_MBatZ} with \PROP{Prop_boom} shows that $\RFH$ is well-defined for any $\vp \in \pls_a.$ \PROP{Prop_MBatZ} gives a recipe for producing from any P-loop a minimally degenerate P-loop, and we can repeat this procedure using a periodic perturbation which is trivial near the integers to obtain a P-loop satisfying the requirements of \DEF{Def_admissible}. \end{proof}

%% file: TransSymp_041013.tex
\section{Transversality and the geometric multiplicity of symplectic eigenvalues}\label{section_TransReview}

\subsection{A review of topological transversality}

Here we recall the basic results concerning finite-dimensional transversality.

\begin{Def} 
Let $M$ and $N$ be smooth manifolds, $A \subset N$ a smooth submanifold, and suppose $\Map{F}{M}{N}$ is a smooth map.  We say that $F$ is \emph{transverse} to $A$ at $x \in F^{-1}(A)$ if
$$\mr{Im} \left(\mr{D}F_x:T_x M \to T_{F(x)}N \right) \, + \, T_{F(x)}A=T_{F(x)}N.$$
If $F$ is transverse to $A$ at $x$ for all $x \in F^{-1}(A)$, we say that $F$ is transverse to $A$ and write $F \pitchfork A.$
\end{Def}

\begin{Thm}[Inverse Function Theorem] 
If $F \pitchfork A$, then $F^{-1}(A)$ is a smooth submanifold of $M$ with dimension 
\beq 
\Dim{F^{-1}(A)}=\Dim{M}+\Dim{A}-\Dim{N},
\eeq 
where a negative dimension is taken to mean that $F$ does not intersect $A$.  In general, for any $\Map{F}{M}{N}$, any submanifold $A$ and any point $x \in F^{-1}(A)$, 
\beq\label{eq_dimcount}
\Dim{ \Set{v}{T_x M}{\DF{F}{v} \in T_{F(x)}A}} \geq \Dim{M}+\Dim{A} - \Dim{N},
\eeq 
with equality holding if and only if $F$ is transverse to $A$ at $x$.
\end{Thm}

\begin{Thm}[Transversality Theorem] 
Let $\Map{F}{M}{N}$ be a smooth map.  Then given a submanifold $A \subset N$ and a generic $F$, $F \pitchfork A.$ 
\end{Thm}

\begin{Def}\label{Def_Lefschetz}
For a real vector space $V$, a linear map $\mathop{T} \in \mr{End}(V)$, and $a \in \R$, define the \emph{$a$-eigenspace} of $\mathop{T}$ by 
\beqn
E_a(T):=\Set{v}{V}{\mathop{T} v = av}
\eeq
Given a smooth manifold $M$, $F \in \mr{Diff}(M)$ and $x \in \mr{Fix}(F),$ 
\beqn
E_a(F,x):= E_a(\mr{D}F_x), \quad i_L(F):=\max_{x \in \mr{Fix}(F)}\Dim{E_1(F,x)}.
\eeq
We call $i_L(F)$ the \emph{Lefschetz degeneracy index} of $F$.  $F$ is \emph{Lefschetz regular} if $i_L(F)=0$ and \emph{Lefschetz degenerate} otherwise.
\end{Def}

\begin{Rmk}
In \DEF{Def_Lefschetz}, $E_a(\mathop{T})$ is defined for all $a \in \R,$ although it is trivial if $a$ is not an eigenvalue.  Our definition of Lefschetz regularity coincides with that used in the Lefschetz fixed point theorem: $F$ is transverse to the diagonal if and only if $i_L(F)=0.$
\end{Rmk}

\begin{Def}\label{Def_dualform} 
Assume that $F\in \mr{Diff}(M)$ is isotopic to the identity, i.e. there exists a smooth map 
$\Map{f}{\R \X M}{M}$ such that
\begin{itemize}
\item For each $t \in \R,$ $f_t \in \mr{Diff}(M)$
\item $f_0$ is the identity map of $M$
\item $f_1=F$.
\end{itemize}
To analyze $f_t$ we examine 
\begin{flalign*}
&\Map{X_t}{[0,1]}{\Ga(TM)}, \quad \mbox{the generating vector field} \\
&\Map{\BAR{F}}{M}{M \X M}, \quad x\mapsto(x,F(x)), \quad\mbox{the graph of $F$}\\
&\Map{\BAR{f}}{[0,1] \X M}{M \X M},  \quad (t,x) \mapsto (x,f_t(x)),\quad \mbox{the parametrized graph of the isotopy}\\
&\Delta_M \subset M\X M, \quad \mbox{the diagonal in $M \X M$}.
\end{flalign*} 
For the remainder of the section, we assume that $X_1 \in \Ga(TM)$ is nonvanishing and fix $\al \in \Om^1(M)$ to be any one form satisfying $\io_{X_1}\al>0.$
\end{Def}

\begin{Lemma}\label{Lem_mandatorydegeneracy}
Suppose $F$, $f_t$, $X_1$ and $\al$ are as in \DEF{Def_dualform}.  Given $x \in \mr{Fix}(F)$, the condition $(F^*\al)(x)=\al(x)$ implies that $\BAR{F}$ is not transverse to the diagonal at $x$.  Moreover, $i_L(F,x)=1$ if and only if $\BAR{f}$ is transverse to the diagonal at $(1,x).$
\end{Lemma}

\begin{proof} 
Examine the map $\BAR{f}.$  By \eqref{eq_dimcount}, for any fixed point $\In{x}{\BAR{F}^{-1}(\Delta_M)}$, there exists a nonzero tangent vector
\beqn
a\p{t}+v \in T_{(1,x)} (\R \X M) \cong \R \oplus T_xM \quad \mbox{with }\DF{\BAR{f}}{(a \p{t}+v)} \in T_{(x,x)} \Delta_M.
\eeq
The Jacobian is found to be
$$\DF{\BAR{f}}{(a \p{t}+v)}=(v,a \, X_1 + \DF{F}{v}).$$
Applying $\al$ to both sides of
\begin{align*}
v &= a \, X_1 + F_* v \\
\intertext{yields}
\dfm{\al}{v} &= a \dfm{\al}{X_1} + \dfm{\al}{F_*v} \\
&= a \dfm{\al}{X_1} + \dfm{(\pb{F}{\al})}{v} \\
&= a \dfm{\al}{X_1} + \dfm{\al}{v}  \\
&\RA a=0, 
\end{align*}
and hence a nonzero vector in $T_xM$ is mapped to the diagonal, i.e. $i_1(F,x)\geq 1.$  By \eqref{eq_dimcount}, $F$ is not transverse to the diagonal.  Moreover, the same equation shows that $i_1(F,x)=1$ if and only if $\BAR{f}$ is transverse to the diagonal at $(1,x).$
\end{proof}

\subsection{Geometric multiplicity of symplectic matrices and symplectomorphisms}\label{section_GenericSymp}
We adapt the previous results to symplectic matrices and symplectomorphisms.
\begin{Def}
Denote by $\sg$ the set of symplectic matrices.  Define subsets 
\bea
\mc{V}_{1,k}&:=\Set{A}\sg{\Dim{E_1(A)}=k}, \quad 0\leq k\leq 2n. \\
\mc{U}_{1,k}&:=\Set{A}\sg{\Dim{E_1(A)} \geq k}, \quad 0\leq k\leq 2n. 
\eea
\end{Def}

\begin{Prop}\label{Prop_GenericSympMat}
$\mc{V}_{1,k}$ is open and dense in $\mc{U}_{1,k}$ for all $0 \leq k \leq n.$  In particular,
\begin{itemize} 
\item The set of all symplectic matrices $A$ such that $1 \notin \si(A)$ is dense in $\sg$
\item $\Set{A}\sg{\Dim{E_1(A)}=1}$ is open and dense in $\Set{A}\sg{1\in\si(A)}.$
\end{itemize}
\end{Prop}

\begin{proof} 
We first examine a neighborhood of the identity in $\sg.$  Recall that the Lie algebra of $\sg$ may be identified with 
\beq
\mathfrak{g}_{\sg} = \Set{JS}{\mc{M}_{2n}(\R)}{S \in \mr{Sym}_{2n}(\R)}.
\eeq
Any symplectic matrix close to the identity may be written as $\exp(JS).$  We first claim that 
\beq
\exp(JS)v=v \LRA Sv=0.
\eeq
This may be seen as follows.  If $B$ is any matrix satisfying $\| B \|<1,$ $(I+B)$ is invertible.  Hence
\bea
\exp(JS)v=v & \LRA \sum_{k=1}^\INF \frac{(JS)^k}{k\!} v =0 \\
&\LRA (I+B)JSv =0,
\eea
where $B$ is determined by $S$.  Hence there exists $\ep>0$ such that $\|S\|<\ep$ implies that $I+B$ is invertible.  Given this universal bound, we have
\beq
E_1(\exp(JS))= \Ker{(1+B)JS} = \Ker{S}.
\eeq
Using the fact that symmetric matrices are diagonalizable, the proof is now straightforward.  To wit, the number of entries equal to zero in a diagonal matrix is evidently upper semi-continuous, which implies that the set of symmetric matrices with nullity $k$ is open and dense in the set of symmetric matrices with nullity greater than or equal to $k+j$, for any $k,$ $j\geq0.$

To apply this reasoning to all of $\sg,$ we look at matrices of the form $A \exp(JS).$  We omit the details of the extension, which are straightforward.
\end{proof}

\begin{Cor} 
A generic \emph{path} $A_t$ of symplectic matrices satisfies $\Dim{E_1(A_t)}\leq 1$ for all $t$.
\end{Cor}
\begin{Cor}
Given a nonzero vector $v \in \R^{2n}$, consider the set of all symplectic matrices 
\beqn
\mathfrak{v}:=\Set{A}\sg{Av=v}.
\eeq
Then given $\ep>0$ and $A \in \mathfrak{v}$, there exists $B \in \mathfrak{v}$ with $\| A-B \| <\ep$ and $E_1(B)=\mr{span}(v).$
\end{Cor}

\begin{Prop}\label{Prop_GenericSymplectomorphism}
The set of symplectomorphisms which are Lefschetz regular is open and dense in $\mr{Symp}(W,\om).$  Further, a generic \emph{path} of symplectomorphisms $\phit$ has $i_L(\phit)\leq 1$ for all $t$.
\end{Prop}
\begin{proof}
For all the statements, the strategy is to appeal to the previous for dimension counts, and then to pass to infinite dimensional parametric transversality to show the transversality can be achieved using symplectomorphisms.  Given $\phi \in \mr{Symp}(M,\om)$, we may regard $\phi_*$ as a section $\phi_* \in \Ga(\Hom{TM,\pb\phi{TM}})$ and we form the map
\beq\label{eq_lift} 
\Map{\widehat{\phi}}{M}{M \X \mr{Hom}\left( TM,\pb{\phi}{TM} \right)}, \quad x \mapsto (\phi(x),\phi_*).
\eeq
Note that $\Hom{TM,\pb{\phi}{TM}}\RE{\mr{Fix}(\phi)}$ has natural subbundles\footnote{Here we regard $\mr{End}(TM)$ as a fiber bundle.  The subbundles we speak of are not linear.} $\mc{V}_{1,k}$ whose fiber above $x$ is the set of endomorphisms whose $1$-eigenspace has dimension $k$.  This is a natural construction when restricted to $\mr{Fix}(\phi)$.\footnote{The reader may compare this construction to the formulation of the Conley Zehnder index. Along a path $u(t)$, trivialize $\Hom{T_{u(0)}M,\pb{u}{TM}}$.  Then the relative Conley Zehnder index is (up to sign) the intersection number of the associated section and $\mc{V}_{1,1}$.}  

Using this language, one observes that a symplectomorphism has Lefschetz degeneracy at most $k$ if and only if the image of $\widehat{\phi}$ does not intersect 
\beqn
\mc{T}_{k+1}:=\mr{Fix}(\phi) \cap \mc{V}_{1,k+1} \subset M \times \Hom{TM,\pb{\phi}{TM}}.
\eeq
Since the codimension of $\mc{T}_{k+1}$ is $2n+k+1,$ we conclude that

\begin{enumerate}
\item  If $\widehat{\phi}$ is transverse to $\mc{T}_*$, it is Lefschetz regular \label{list_transone}
\item  If a path of symplectomorphisms is transverse to $\mc{T}_*$, then $\phit$ has Lefschetz degeneracy at most $1$ for all $t.$ \label{list_transtwo}
\end{enumerate}
Here we employ standard usage that a map is transverse to a stratified submanifold if and only if it is transverse to all the strata.  It is clear that the respective sets satisfying (\ref{list_transone}), (\ref{list_transtwo}) are open.  To show that they are dense, we must pass to parametric transversality.

The question we now address is the following: given a symplectomorphism $\phi$, it has been established that there exist \emph{diffeomorphisms} arbitrarily close to $\phi$ enjoying the desired properties.  Can we strengthen this to symplectomorphisms?  This is far more standard, hence we only sketch the argument.  Form the evaluation map 
\beqn \mr{Symp}(M,\om) \X M \to M\X (T^*M \otimes_{\mr{Symp}(M,\om)}TM),\quad (\psi,x) \mapsto (\psi(x),\psi_*). \eeq
Since this map is transverse to $\mc{T}_*$, the inverse function theorem guarantees that a generic symplectic map will satisfy the requirement.
\end{proof}